\documentclass[10pt,twoside]{siamltex}
\usepackage{amsfonts,epsfig}

\newtheorem{remark}[theorem]{Remark}

\newtheorem{example}[theorem]{Example}

\begin{document}

\bibliographystyle{plain}
\title{
The structure of strong linear preservers of gw-majorization on $
\mathbf{M}_{{n,m}}$\thanks{Received by the editors on ... Accepted
for publication on ....   Handling Editor: ...}}

\author{
A. Armandnejad\thanks{Department of Mathematics, Valiasr
University of Rafsanjan, 7713936417, Rafsanjan, Iran.\ \ \ \ \ \ \
 armandnejad@mail.vru.ac.ir}
\and A. Salemi\thanks{Department of Mathematics, Shahid Bahonar
University of Kerman, 7619614111, Kerman, Iran.
salemi@mail.uk.ac.ir}}

\pagestyle{myheadings} \markboth{A. Armandnejad and A. Salemi}{The
structure of strong linear preservers of gw-majorization on $
\mathbf{M}_{{n,m}}$} \maketitle

\begin{abstract}
Let $\mathbf{M}_{n,m}$ be the set of all $n\times m$
  matrices with entries in $\mathbb{F}$, where $\mathbb{F}$ is the field of real or complex numbers.
  A matrix $R\in M_n$ with the property Re=e,
 is said to be a g-row stochastic (generalized row stochastic) matrix. Let A,B$\in \mathbf{M}_{n,m}$,
  so B is said to be  gw-majorized by A if
  there exists an n$\times $n g-row stochastic
matrix R such that B=RA. In this paper we characterize all linear
operators that  strongly preserve gw-majorization on $
\mathbf{M}_{n,m}$ and all linear operators that  strongly preserve
matrix majorization on $ \mathbf{M}_{n}$.
\end{abstract}

\begin{keywords}
Preserver, strong preserver, g-row stochastic
  matrices, gw-majorization.
\end{keywords}
\begin{AMS}
15A03, 15A04, 15A51.
\end{AMS}

\section{Introduction} A nonnegative matrix $R\in M_n$ with the property Re=e,
 is said to be a row stochastic matrix. Let A,B$\in \mathbf{M}_{n,m}$,
  so B is said to be  matrix-majorized by A if there exists an n$\times $n row stochastic
matrix R such that B=RA. The definition of matrix majorization was
introduced by Dahl in \cite{D}. For more information about
majorization see \cite{Bhatia} and \cite{marshal}.

 Let $\sim$ be a relation on $\mathbf{M}_{n,m}$ .
 A linear operator T:$\mathbf{M}_{n,m}\longrightarrow\mathbf{M}_{n,m}$
  is said to be a linear
strong preserver of $\sim$ whenever:
\begin{eqnarray} \nonumber
 x\sim y \ \ \ \Longleftrightarrow \ \ \ T(x)\sim T(y) .
 \end{eqnarray}
  A  matrix $D\in M_n$ with the properties De=e and $D^t$e=e,
 is said to be a g-doubly stochastic matrix. Let A,B$\in \mathbf{M}_{n,m}$,
  so B is said to be  gs-majorized by A if there exists an n$\times $n g-doubly stochastic
matrix D such that B=DA. The definition of gs-majorization was
introduced in \cite{A.S3} and authors proved that a linear
operator $T:\mathbf{M}_{n, m}\rightarrow \mathbf{M}_{n, m}$
strongly preserves gs-majorization if and only if $T(X)=AXR+JXS$
for some $R,S\in \mathbf{M}_m$ and $A\in\mathbf{M}_n$,  such that
A, $R$ and $R+nS$ are invertible and  A is g-doubly stochastic .

 In \cite{B.L.L}, Beasley, S.-G. Lee and
Y.H Lee proved that, if a linear operator
T:$\mathbf{M}_{n}$$\rightarrow
$$\mathbf{M}_{n}$  strongly preserves matrix majorization then,
 there exist a permutation P and an invertible matrix
$M\in\mathbf{M}_n$ such that $T(X)=PXM$  for every  X in
$span\{\mathbf{R}_n\}$, where $\mathbf{R}_n$ is the set of all
$n\times n$ row stochastic matrices, and currently  A.M. Hasani
and M. Radjabalipour in [\ref{radjab}] showed that:
\begin{eqnarray}\label{rad}
 T(X)=PXM,
\forall \ X\in \ \mathbf{M}_n.
\end{eqnarray}
In \cite{A.S2}  authors introduced gw-majorization and
characterized its strong linear preservers on $\mathbf{M}_n$.  In
this paper, we will to show that a linear operator
$T:\mathbf{M}_{n,m}\rightarrow \mathbf{M}_{n,m}$
 strongly preserves gw-majorization if and only if $T(X)=AXB$ for every  X in $\mathbf{M}_{n,m}$,
where $A\in\mathbf{GR}_n$ and $B\in \textbf{M}_m$ are invertible
matrices. In the end  we state a corollary that regains
(\ref{rad}).

 Throughout  this paper, $\mathbf{GR}_n$ is the set of all g-row stochastic
 matrices , e=$(1,...,1)^t$$\in \mathbb{F}^n$ and J=$ee^t$$\in \mathbf{M}_n$.
\section{Strong linear preservers of gw-majorization on $\mathbf{M}_{n,m}$}
In this section  we state some properties of gw-majorization on
$\mathbf{M}_{n,m}$ then we characterize all linear operators on
 $\mathbf{M}_{n,m}$ that strongly preserve gw-majorization.

A matrix $R\in \textbf{M}_n$ with the property Re=e, is said to be
a g-row stochastic matrix. For more details
 see \cite{C.L.}.
\begin{definition}Let $A, B\in \mathbf{M}_{n,m}$.
 The matrix B is said to be  gw-majorized by A if
  there exists an n$\times $n g-row stochastic
matrix R such that B=RA and denoted by A$\succ_{gw}$B .

\end{definition}

\begin{proposition}\label{inv} Let
T :$\mathbf{M}_{n,m}$$\rightarrow $$\mathbf{M}_{n,m}$ be a linear
operator that strongly preserves gw-majorization . Then T is
invertible.
\end{proposition}
\begin{proof}
Suppose T(A)=0. Since T is linear and 0$\succ_{gw} $T(A), T(0)
$\succ_{gw}$T(A). Therefore,  0$\succ_{gw}$A because T strongly
preserves gw-majorization. Then, there exists an n$\times$n g-row
stochastic matrix R such that A=R0. Then, A=0 and hence T is
invertible.
\end{proof}
\begin{remark}\label{close}Let A,B be two g-row  stochastic matrices
then, AB  and $A^{-1}$ (If A is invertible) are g-row stochastic
matrices.
\end{remark}

 The relation gw-majorization on
 $\mathbf{M}_{n,m}$ has the following properties :\\
Let $X,Y\in \mathbf{M}_{n,m}$, $A,B\in \mathbf{GR}_n$, $C\in
\mathbf{M}_m$ and $\alpha ,\beta\in \mathbb{F}$ such that A,B and
C are invertible and $\alpha\neq0$. Then the following conditions
are equivalent:

1. \ $X\succ_{gw}Y$

2. \ $AX\succ_{gw}BY$

3. \ $\alpha X+\beta J_{n,m}\succ_{gw}\alpha Y+\beta J_{n,m}$

4. \ $XC\succ_{gw}YC$

Where $J_{n,m}$ is the $n\times m$ matrix whose all entries are
equal one.

Now, we characterize the linear preservers of gw-majorization on
$\mathbb{F}^n$.

 \begin{lemma}\label{lem1}
Let $x\in\mathbb{F}^n$. Then $x\succ_{gw}y$, $\forall
y\in\mathbb{F}^n$ if and only if
 $x\notin span\{e\}$.
\end{lemma}
\begin{proof}
Let $x\succ_{gw} y$, $\forall y\in\mathbb{F}^n$, it is clear that
$x\notin span\{e\}$. Conversely, let $x=(x_1,\cdots,x_n)^t\notin
span\{e\}$, then x has at least two distinct components such as
$x_{k}$ and $x_{l}$. Let $y=(y_1,\cdots,y_n)^t\in\mathbb{F}^n$ be
arbitrary, for $1\leq i,j\leq n$ define

$r_{ik}= \frac{y_i-x_l}{x_k-x_l} $ , $r_{il}=
\frac{-y_i+x_k}{x_k-x_l}$ and $r_{ij}=0$ If $ j\neq k,l$. Then
$R=(r_{ij})\in \mathbf{GR}_{n}$ and $ Rx=y$, so $ x \succ_{gw} y$.
\end{proof}
\begin{lemma}\label{lem2}
Let $T:\mathbb{F}^n \rightarrow \mathbb{F}^n$ be a non zero linear
operator. Then T preserves  gw-majorization if and only if
$x\notin span\{e\}$ implies that $T(x)\notin span\{e\}$.
\end{lemma}
\begin{proof}
Let T preserves  gw-majorization. Assume that $x\notin span\{e\}$,
then $x\succ_{gw} y$, $\forall y \in\mathbb{F}^n$ by Lemma
\ref{lem1}. Therefore $T(x)\succ_{gw}T(y)$, $\forall
y\in\mathbb{F}^n$.
 If $T(x)\in span\{e\}$ then T=0, a contradiction, so
$T(x)\notin span\{e\}$.

Conversely, let $x\notin span\{e\}$ implies that $T(x)\notin
span\{e\}$. If $x\succ_{gw}y$ then we have two cases:

 Case 1; Let $x\in span\{e\}$, then  x=y and hence
T(x)=T(y).

  Case 2; Let $x\notin span\{e\}$, then $T(x)\notin
span\{e\}$ by hypostasis, so by Lemma \ref{lem1},
$T(x)\succ_{gw}Z$, $\forall Z\in\mathbb{F}^n$ and hence
$T(x)\succ_{gw}T(y)$. Then T preserves  gw-majorization.
\end{proof}

\begin{theorem}\label{preserve}
Let $T:\mathbb{F}^n \rightarrow \mathbb{F}^n$ be a linear
operator. Then T preserves gw-majorization if and only if $
T(x)=\alpha Rx $ for some $R\in \mathbf{M}_n$  and
$\alpha\in\mathbb{F}$, such that either  $ker(R)=span\{e\}$ and
$e\notin Im(R)$ or $R\in \mathbf{GR}_n$ is invertible.
\end{theorem}
\begin{proof}
 If T =0, we put $\alpha=0$. Let $T\neq 0$ and A be the matrix representation of T with
 respect to the standard basis of $\mathbb{F}^n$. Now, we consider two
cases:

 Case 1; Let T be invertible. Then there exits
$b\in\mathbb{F}^n$ such that Ab=e. So b=re, for some
$r\in\mathbb{F}$, by Lemma $\ref{lem2}$. Then $Ae=\frac{1}{r}e$,
therefore $T(x)= \alpha R x$, where $\alpha=\frac{1}{r}$ and
$R=(rA)\in\textbf{GR}_n$ is invertible.

 Case 2; Let T be singular.
Then by Lemma $\ref{lem2}$, ker(T)= span$\{e\}$ and $e\notin
Im(T)$. So ker(A)=span$\{e\}$ and $e\notin Im(A)$. The converse is
trivial.
\end{proof}

Now, we state the following two Lemmas to prove the main Theorem
of this paper.
 \begin{lemma}\label{ker A=e}
 Let $A\in \textbf{M}_n$ be such that ker(A)=$span\{e\}$. Then there exist
 $x_0,y_0\in\mathbb{F}^n$ and $R_0\in\mathbf{GR}_n$ such that
 $x_0+Ay_0$ doesn't gw-majorize $R_0x_0+AR_0y_0$.
\end{lemma}
\begin{proof}
Assume if possible,
\begin{eqnarray}\label{1}
x+Ay\succ_{gw}Rx+ARy, \forall x,y\in \mathbb{F}^n, \forall
R\in\mathbf{GR}_n \ \ .
\end{eqnarray}
Now, we consider two cases:

 Case 1; Let $e\in Im(A)$, then there
exists $y_0\in\mathbb{F}^n$, such that $Ay_0=e$. Put $x=0,y=y_0$
in $(\ref{1})$ then $ARy_0=e,\forall R\in\mathbf{GR}_n$, a
contradiction .

Case 2; Let $e\notin Im(A)$, then $\mathbb{F}^n=Im(A)\bigoplus
span\{e\}$. So for every i ($1\leq i\leq n$), there exist
$y_i\in\mathbb{F}^n$ and $r_i\in\mathbb{F}$ such that
$e_i=Ay_i+r_ie$. Put $x=e-(e_i- r_ie)$ and $y=y_i$ in $(\ref{1})$,
then
\begin{eqnarray}\label{2}
r_ie-Re_i+ARy_i=0, \ \forall R\in\mathbf{GR}_n \ \ .
\end{eqnarray}
For every j ($1\leq j\leq n , j\neq i$) put $R_j=ee^t_j$ in
$(\ref{2})$, then $r_i=0$, for every i ($1\leq i\leq n$).
Therefore $Ay_i=e_i$, for every i ($1\leq i\leq n$), then
$Im(A)=\mathbb{F}^n$, a contradiction.
\end{proof}

\begin{lemma}\label{RA=AR}
Let A$\in \mathbf{GR}_n$ be invertible. Then the following
conditions are equivalent:

(a) A=I

 (b) $(x+Ay)\succ_{gw}(Rx+ARy)$, $\forall R\in
$$\mathbf{GR}_n$ and $\forall x,y\in \mathbb{F}^n$ .
\end{lemma}
\begin{proof}
 It is clear that, (a) implies (b).
Conversely, let (b) holds. The matrix A is invertible, then for
every i ($1\leq i\leq n$) there exists $y_i\in \mathbb{F}^n$ such
that $Ay_i=e-e_i$ . By hypostasis $(e_i+Ay_i)\succ_{gw}
(Re_i+ARy_i)$, $\forall  R\in \mathbf{GR}_n$, then
\begin{eqnarray} \label{3.5}
(Re_i+ARy_i)=e, \ \forall R\in \mathbf{GR}_n\ \ .
\end{eqnarray}
For every $R\in \mathbf{GR}_n$, it is clear that  $R[J-(n-1)A]\in
\mathbf{GR}_n$,  therefore by (\ref{3.5}),
\begin{eqnarray} \nonumber
  R[J-(n-1)A]e_i+AR[J-(n-1)A]y_i=e& \Rightarrow &
  (RA-AR)e_i=0,\ \forall \ i \in \{1,..., n\}\\
    \nonumber & \Rightarrow & AR=RA,\forall \ R\in \mathbf{GR}_n.
     \end{eqnarray}
So it is easy to show that A=I.
 \end{proof}

 Now, we state the main Theorem of this paper.
\begin{theorem}\label{main}
Let $T:\mathbf{M}_{n,m}\rightarrow \mathbf{M}_{n,m}$ be a linear
operator. Then T strongly preserves gw-majorization if and only if
$T(X)=AXB$ for every $X\in\mathbf{M}_{n,m}$ , where
$A\in\mathbf{GR}_n$ and $B\in \textbf{M}_m$ are invertible .
\end{theorem}
 \begin{proof}
 If m=1, the result is implied by Theorem $\ref{preserve}$, so let $m\geq2$.
 Define the embedding
$E^j$:$\mathbb{F}^n\rightarrow \mathbf{M}_{n,m}$ by
$E^j(x)=xe_j^t$ and projection $E_i$ :$
\mathbf{M}_{n,m}\rightarrow \mathbb{F}^n$ by $E_i(X)=Xe_i$ for
every $i,j \in \{1,..., m\}$. Put $T_i^j=E_i TE^j$ and let
$X=\textbf{[}x_1|\cdots|x_m\textbf{]}\in \mathbf{M}_{n,m}$ where
$x_i$ is the $i^{th}$ column of X . Then,
\begin{eqnarray}\nonumber
T(X)=T(\textbf{[}x_1|\cdots|x_m\textbf{]})=\textbf{[}\sum_{j=1}^m
T_1^j(x_j)|\cdots|\sum_{j=1}^mT_m^j(x_j)\textbf{]}.
\end{eqnarray}
It is easy to show that
$T_i^j$:$\mathbb{F}^n\rightarrow\mathbb{F}^n$ preserves
gw-majorizaton. Then by Theorem \ref{preserve}, there exist
$\alpha_i^j\in\mathbb{F}$, and $A_i^j\in \textbf{M}_n$ such that
$T_i^j(x)=\alpha_i^j A_i^j x$ where either $A_i^j\in\mathbf{GR}_n$
is invertible or $ker(A_i^j)= span\{e\}$ and $e\notin Im(A_i^j)$.
Then,
\begin{eqnarray}\label{0}
T(X)=\textbf{[}\sum_{j=1}^m \alpha_i^j A_i^j x_j
|\cdots|\sum_{j=1}^m \alpha_m^j A_m^j x_j\textbf{]}.
\end{eqnarray}
Now, we consider three steps for the proof.

Step 1. In this step we will to show that, if there exist p and q
($1\leq p,q\leq m$) such that $\alpha_p^q\neq0$ and $
A_p^q\in\mathbf{GR}_n$ is invertible, then for every j ($1\leq
j\leq m)$, $A_p^j=A_p^q$. If $\alpha_p^j=0$, without lose of
generality
 we can choose $A_p^j=A_p^q$ .
Let $\alpha_p^j\neq0$. For every $x,y\in\mathbb{F}^n$, put
$X=xe_q^t+ye_j^t$, then $T(X)\succ_{gw}T(RX),\forall
R\in\mathbf{GR}_n$ and hence by (\ref{0}),
\begin{eqnarray}\nonumber
\alpha_p^qA_p^qx+\alpha_p^jA_p^jy\succ_{gw}\alpha_p^qA_p^qRx+\alpha_p^jA_p^jRy
, \forall x,y\in\mathbb{F}^n , \forall R\in\mathbf{GR}_n&\Rightarrow&\\
 \nonumber\\
 \nonumber
x+(A_p^q)^{-1}A_p^j(\frac{\alpha_p^j}{\alpha_p^q}y)\succ_{gw}Rx+(A_p^q)^{-1}A_p^jR(\frac{\alpha_p^j}{\alpha_p^q}y)
,\forall x,y\in\mathbb{F}^n , \forall R\in\mathbf{GR}_n&\Rightarrow&\\
\nonumber
\\
\nonumber x+(A_p^q)^{-1}A_p^jy\succ_{gw}Rx+(A_p^q)^{-1}A_p^jRy
,\forall x,y\in\mathbb{F}^n , \forall R\in\mathbf{GR}_n.
\end{eqnarray}
So by Lemma \ref{ker A=e}, $A_p^j$ is invertible and hence by
Lemma \ref{RA=AR}, $A_p^j=A_p^q$. Set $A_p=A_p^q$, then
 \begin{eqnarray}\nonumber
 T(X)=\textbf{[}\sum_{j=1}^m \alpha_1^j
A_1^jx_j|\cdots|A_p\sum_{j=1}^m\alpha_p^jx_j|\cdots|\sum_{j=1}^m\alpha_m^jA_m^jx_j\textbf{]}.
\end{eqnarray}

Step 2. In this step we will to show that for every i and j
($1\leq i,j\leq m$), $A_i^j\in\mathbf{GR}_n$ is invertible if
$\alpha_i^j\neq0$. Assume if possible there exist r and s ($1\leq
r,s\leq m)$, such that $ker (A_r^s)=span\{e\}$ and
$\alpha_r^s\neq0$. Without lose of generality we can assume that
r=m, then by step 1, for every $1\leq j\leq m$,
$ker(A_m^j)=span\{e\}$. Now, we construct a non zero $n\times m$
matrix U, such that T(U)=0. Consider the vectors:
\begin{eqnarray}\nonumber
b_1=\left(%
\begin{array}{ccc}
 \alpha_1^1 \\
 \vdots\\
 \alpha^1_{m-1} \\
\end{array}%
\right) ,\cdots,
b_m=\left(%
\begin{array}{ccc}
  \alpha_1^m \\
  \vdots\\
  \alpha^m_{m-1} \\
\end{array}%
\right) \in \mathbb{F}^{m-1}.
\end{eqnarray}
 It is clear that
$\{b_1,\cdots,b_m\}$ is a linearly dependent set in
$\mathbb{F}^{m-1}$, so there exist (not all zero)
$\lambda_1,\cdots,\lambda_m\in\mathbb{F}$, such that
\begin{eqnarray}\nonumber
 \sum_{j=1}^m\lambda_j\alpha_i^j=0\ ,\ \forall \ i\in\{1,..., m-1\} .
\end{eqnarray}
 Now, define
$U:=\textbf{[}\lambda_1e|\cdots|\lambda_me\textbf{]}\in\mathbf{M}_{n,m}$.
It is clear that, $U\neq0$ and
\begin{eqnarray}\nonumber
T(U)=\textbf{[}\sum_{j=1}^m\lambda_j\alpha_1^jA_1^je|\cdots|\sum_{j=1}^m\lambda_j\alpha_m^jA_m^je\textbf{]}.
\end{eqnarray}
We will show that T(U)=0. Since $ker(A_m^j)=span\{e\}$, it is
clear that $\sum_{j=1}^m\lambda_j\alpha_m^jA_m^je$\\=0 and hence
the last column of T(U) is zero. Now, for every k ($1\leq k\leq
m-1$), we consider the $k^{th}$ column of T(U):

Case 1; Let $\alpha_k^l\neq0$ and $A_k^l\in\mathbf{GR}_n$ be
invertible for some $l$ ($1\leq l\leq m$), then by step 1 ,
\begin{eqnarray}\nonumber
 \sum_{j=1}^m\lambda_j\alpha_k^jA_k^je=A_k^l(\sum_{j=1}^m\lambda_j\alpha_k^j)e=0 .
\end{eqnarray}

 Case 2; Let for every j ($1\leq j\leq m$), $A_k^j$ be non
invertible, then $ker(A_k^j)=span\{e\}$, so
$\sum_{j=1}^m\lambda_j\alpha_k^jA_k^je=0$. Therefor T(U)=0, a
contradiction. So by step 1 there exist invertible matrices $A_i
\in \mathbf{GR}_n$ $(1\leq i\leq m)$ such that
$T(X)=T\textbf{[}x_1|\cdots|x_m\textbf{]}
 \\ =\textbf{[}A_1Xa_1|\cdots|A_mXa_m\textbf{]}$, where
$a_i=(\alpha_i^1,\cdots,\alpha_i^m)^t$, for every i ($1\leq i\leq
m$)
 .

Step 3. In this step we will to show that $A_i=A_1$ , for all
$1\leq i\leq m$. Now, we show that
$rank\textbf{[}a_1|...|a_m\textbf{]}\geq 2$. Assume if possible,
$\{a_1,...,a_m\}\subseteq span\{a\}$, for some $a\in\mathbb{F}^m$.
Since $m\geq 2$, then we choose $
b\in(span\{a\})^\bot\setminus\{0\}$. Define $X_0:=e_1b^t\in
\mathbf{M}_{n,m}$.
 It is clear that $X_0\neq0$ and $T(X_0)=0$, a contradiction and
 hence
$rank\textbf{[}a_1|...|a_m\textbf{]}\geq 2$ .
 Without lose of generality we can assume that $\{a_1,a_2\}$ is a linearly
independent set. Let $X\in \mathbf{M}_{n,m}$ and $R\in
\mathbf{GR}_n$ be arbitrary, then
 \begin{eqnarray} \label{gg} \nonumber
 X\succ_{gw}RX &\Rightarrow& T(X)\succ_{gw}T(RX)
  \\ \nonumber &\Rightarrow&
  \textbf{[}A_1X a_1|...|A_mX a_m\textbf{]}\succ_{gw}\textbf{[}A_1 RX a_1|...|A_m RX a_m\textbf{]}
\\ \nonumber &\Rightarrow&
   A_1X a_1+A_2X a_2\succ_{gw}A_1 RX a_1+A_2 RX a_2
 \\ &\Rightarrow&
 Xa_1+(A_1^{-1} A_2)X a_2\succ_{gw}RXa_1+(A_1^{-1} A_2) RX a_2 \ \ .
     \end{eqnarray}

 Since $\{a_1,a_2\}$ is linearly independent, then for every
 $x,y \in \mathbb{F}^n$, there exits $B_{x,y}\in \mathbf{M}_{n,m}$
 such that, $B_{x,y}\ a_1=x$,\ \ \ $B_{x,y} \ a_2=y$, put
 $X=B_{x,y}$ in (\ref{gg}) thus,
 \begin{eqnarray} \nonumber
  B_{x,y} \ a_1+(A_1^{-1}\  A_2 )\ B_{x,y}\  a_2 &\succ_{gw}& RB_{x,y}
  \ a_2+(A_1^{-1}\  A_2) \ RB_{x,y}\
 a_2 \Rightarrow \\x+(A_1^{-1} A_2) \ y &\succ_{gw}&\nonumber Rx +(A_1^{-1}
 A_2)Ry\ , \forall R\in\mathbf{GR}_n \ .
 \end{eqnarray}
 Then by Lemma $\ref{RA=AR}$, \ $A_1^{-1} A_2=I$ and hence
  $A_2=A_1$ .
 For every i ($3\leq i\leq m$), if $a_i=0$  we
 can replace  $A_i$ by $ A_1$. If $a_i\neq 0$, then
 $\{a_1,a_i\}$ or $\{a_2,a_i\}$ is a linearly independent set. By the
 same method as above,
 $A_i=A_1 $ or $A_i=A_2$.
  Let $ A=A_1$ and hence $A_i=A$ for
 every i ($1\leq i\leq m$)  . Therefore,
 \begin{eqnarray} \nonumber
 T(X)=\textbf{[} AXa_1 \  | \cdots | AXa_m\textbf{]}= AXB
 ,\end{eqnarray} where $B=\textbf{[}  a_1|  \cdots |\ a_m \textbf{]}$ is an
 invertible matrix in $\textbf{M}_m$.

 Conversely, if T(X)=AXB where $A\in\mathbf{GR}_n$ and $B\in
\textbf{M}_m$ are invertible matrices, it is trivial that T
strongly preserves gw-majorization.
\end{proof}

The following statement shows that every strong linear preserver
of matrix majorization is an strong linear preserver of
gw-majorization but the converse is false.
\begin{proposition} \label{implygw}
Let $T:\mathbf{M}_{n,m}\rightarrow \mathbf{M}_{n,m}$ be a linear
operator that strongly preserves matrix majorization.
 Then T strongly preserves gw-majorization .
 \end{proposition}
\begin{proof}
Let $A\succ_{gw}B$. Then there exists a g-row stochastic matrix R
such that B=RA. For the g-row stochastic matrix R, there exist
scalars $r_1 ,...,r_k$ and row stochastic matrices $R_1 ,..., R_k$
such that $\sum_{i=1}^k r_i=1 $ and $R=\sum_{i=1}^k r_i R_i$. For
every i $(1\leq i\leq k)$, $A\succ_{}R_i A$ and hence
$T(A)\succ_{}T(R_i A)$. Then there exist row stochastic matrices
$S_i$ $(1\leq i\leq k)$, such that $T(R_i A)=S_i T(A)$. Put
S=$\sum_{i=1}^k r_i S_i$ , it is clear that S is a g-row
stochastic matrix and $T(B)=S T(A)$. Therefore
$T(A)\succ_{gw}T(B)$. For other side replace T by $T^{-1}$ and
similarly conclude that $A\succ_{gw}B$ where $T(A)\succ_{gw}T(B)$.
Then T strongly preserves gw-majorization.
\end{proof}
\begin{example}
Let the linear operator $T:\mathbf{M}_{2}\rightarrow
\mathbf{M}_{2}$ be such that
T(X)=AX, where $A=\left(%
\begin{array}{cc}
  1 & 0\\
  -1 & 2 \\
\end{array}%
\right)$. It is clear that T strongly preserves gw-majorization by
Theorem \ref{main} . But T doesn't strongly preserve matrix
majorization. For this cosider the following matrices:

 $\left(%
\begin{array}{cc}
  1 & 0 \\
  0 & 0 \\
\end{array}%
\right)\ \ \  and \ \ \left(%
\begin{array}{cc}
  0 & 0 \\
  1 & 0 \\
\end{array}%
\right).$
\end{example}

Now we state the following corollary that characterize all linear
operator that strongly preserve matrix majorization on
$\mathbf{M}_{n}$.
\begin{corollary}[Theorem 5.2 ,\ref{radjab}]
A linear operator $T:\mathbf{M}_{n}\rightarrow \mathbf{M}_{n}$
 strongly preserves matrix majorization  $\succ$  if and only if
$T(X)=PXL$, where P is permutation and $L\in \textbf{M}_n$ is
invertible .
\end{corollary}
\begin{proof}
Let T strongly preserves matrix majorization. Then T strongly
preserves gw-majorization by Proposition \ref{implygw}. Therefore
in view of  Theorem \ref{main}  there exist invertible matrices
$A\in \textbf{GR}_n$ and $B\in \textbf{M}_n$ such that $T(X)=AXB$
for all $X\in \textbf{M}_n$ . For every  row stochastic matrix R,
it is clear that $I\succ R$.
 So $T(I)\succ T(R)$ for every row stochastic matrix R. Then
$AIB\succ ARB$ and hence $ RA^{-1}$ is a row stochastic matrix,
for every row stochastic matrix R. It is easy to show that
$A^{-1}$ is a row stochastic matrix. Similarly A is a row
stochastic matrix too and hence A is a permutation matrix.
\end{proof}


\bigskip



\begin{thebibliography}{}

\bibitem{A.S3}
A. Armandnejad, A. Salemi, The structure  of linear
 preservers of gs-majorization.
 \newblock  {\em Bull. Iranian Math. Soc}, Vol.32
 No.2 (2006) 31-42 .

 \bibitem{A.S2}
A. Armandnejad, A. Salemi, Strong linear
 preservers of gw-majorization.
 \newblock  {\ Journal Of Dynamical
 Systems and Geometric Theories}, Vol.6 (2007) .
\bibitem{B.L.L}
L.B.Beasley, S.-G. Lee, Y.H Lee.
\newblock  A characterization of strong preservers of matrix
majorization.
\newblock  {\em Linear Algebra and Its Applications}, 367:341-346, 2003.

\bibitem{C.L.}
H. Chiang and C.K. Li.
\newblock  Generalized Doubly Stochastic Matrices and Linear
Preservers.
\newblock  {\em Linear and Multilinear Algebra}, 53:1-11, 2005.

\bibitem{Bhatia}
R. Bhatia.
\newblock  Matrix Analysis.
\newblock {\em Springer-Verlag, New York}, 1997.
\bibitem{D}
G. Dahl.
\newblock  Matrix majorization.
\newblock {\em Linear Algebra Appl} , 288:53-73 , 1999.

\bibitem{}\label{radjab} A.M. Hasani and M.
Radjabalipuor,The structure of linear operators strongly
preserving majorizations of matrices.
\newblock  {\em Electronic Journal of Linear
Algebra}, 15(2006), 260-268.

\bibitem{marshal}
A.W. Marshall, I.  Olkin,
\newblock   Theory of Majorization and its Applications.
\newblock  {\em Academic, New York}, 1972.



\end{thebibliography}
\end{document}